\documentclass[IJAAMMM,PDF]{ijaamm} 
\usepackage{layout}
\usepackage{hyperref}
\usepackage{epstopdf}
\usepackage{amssymb}

\title{Special Smarandache Curves with Respect to Darboux Frame in Galilean 3-Space}

\articletype{Research Article} 

\author{Tevfik \c{S}ahin\inst{1,} \email{tevfik.sahin@amasya.edu.tr, tevfiksah@gmail.com},
        Merve Okur\inst{2}
       }

\shortauthor{Tevfik \c{S}ahin et. al.}

\institute{\inst{$^1$}
           Department of Mathematics, Faculty of Arts and Sciences, Amasya University, 05000, Amasya, Turkey\\
           Email:tevfik.sahin@amasya.edu.tr, tevfiksah@gmail.com
           \inst{$^2$}
           Institute of Science, Department of Mathematics, 05000, Amasya, Turkey\\
           Email: okurmerwe869@gmail.com
}

\abstract{In the present paper, we investigate special Smarandache curves with Darboux apparatus with respect to Frenet and Darboux frame of an arbitrary curve on a surface in the three-dimensional Galilean space $G^{3}$. Furthermore, we give general position vectors of special Smarandache curves of geodesic, asymptotic and curvature line on the surface in $G^{3}$. As a result of this, we provide some related examples of these curves.
}



\begin{document}

\maketitle

\section{Introduction }

For centuries, it was thought that Euclidean geometry is the only geometric system until the discoveries of hyperbolic geometry that is a non-Euclidean geometry.  In 1870,  it was shown by Cayley-Klein that there are 9 different geometries in the plane including Euclidean geometry.  These geometries are determined by parabolic, elliptic, and hyperbolic measures of angles and lengths. 
The main aim of this work is to study some special curves in Galilean geometry which is also among foregoing geometries.  The conventional view about Galilean geometry is that it is relatively simpler than Euclidean geometry. There are some problems that cannot be solved in Euclidean geometry, however  they are an easy matter in Galilean geometry. For instance,  the problem of determination of position vector of an arbitrary curve and  the problem that we study in this article can be considered as good examples for the case.  Another advantageous of Galilean geometry is that it is associated with the Galilean principle of relativity.  For more details about Galilean geometry, we refer the interested reader to the book by Yaglom \cite{yag}.

The theory of curves forms an important and useful class of theories in differential geometry. The curves emerge from the solutions of some important physical problems. Also, mathematical models are often used to describe complicated systems that arising in many different branch of science such as engineering, chemistry, biology, etc. \cite{biyo, kimya}

A curve in space is studied by assigning at each point a moving frame. 
The method of moving frame is a central tool to study a curve or a surface. The fundamental theorem of curves states that curves are determined by curvatures and Frenet vectors \cite{krey}. Thus, curvature functions provide us with some special and important information about curves. For example; line, circle, helix (circular or generalized), Salkowski curve, geodesic , asymptotic and line of curvature etc.  All of these curves are characterized by the specific conditions imposed on their curvatures. To examine the characteristics of this curves, it is important that the position vectors of the curves are given according to the curvature functions. However, this is not always possible in all geometries. For example, the problem of determination of the position vector of a curve in Euclidean or Minkowski spaces can only be solved for some special curve such as plane line, helix and slant helix.
However, this problem can  be solved independent of type of curves in Galilean space \cite{ali,buket}. 

Curves  can also be produced in many different ways, such as solution of physical problems,  trajectory of a moving particle,  etc. \cite{krey}. In addition,  one can produce a new curve by using Frenet vector fields of a given curve, such as Evolutes and involutes, spherical indicatrix, and Smarandache curves.

If the position vector of $\alpha$ curve is formed by frame vectors of $\beta$ curve, then $\alpha$ is called Smarandache curve of $\beta$ \cite{suha}.
Recently, many researchers have studied special Smarandache curves with respect to different frames in different spaces. In \cite{suha}, the authors introduced a special case of Smarandache curves in the space $E_4^1$. \cite{ali10} studied special Smarandache curve in Euclidean space $E^3$. In \cite{yuce, cetin}, the authors investigate the curves with respect to Bishop and Darboux frame in $E^3$, respectively. Also, \cite{saad17} investigated the curves with respect to Darboux frame in Minkowski $3-$space.

Among these studies, only \cite{saad} used general position vector with respect to Frenet frame of curve to obtain Samarandache curves in Galilean space.

The main aim of this paper is to determine position vector of Smarandache curves of arbitrary curve on a surface in $G_3$ in terms of geodesic, normal curvature and geodesic torsion with respect to the standard frame. The results of this work include providing Smarandache curves of some special curves such as geodesic, asymptotic curve, line of curvature on a surface in $G_3$ and Smarandache curves for special cases of curves  such as, Smarandache curves of geodesics that are circular helix, genaralized helix or Salkowski, etc. Finally, we elaborate on some special curves by giving their graphs.  

\section{Introduction and Preliminaries}

The Galilean space $G^{3}$ is one of the Cayley-Klein spaces associated with the projective metric of signature $\left( 0,0,+,+\right) $ \cite{mol}. The absolute figure of the Galilean space is the ordered triple $\{w,f,I\}$, where $w$ is an ideal (absolute) plane, 
$f$  is a line (absolute line) in $w$, and $I$ is a fixed eliptic involution of points of $f$. 

%

In non-homogeneous coordinates the group of isometries of $G^{3}$ has the following form:
\begin{eqnarray}
	\overline{x} &=&a_{11}+x,  \notag \\
	\overline{y} &=&a_{21}+a_{22}x+y\cos \varphi +z\sin \varphi , \\
	\overline{z} &=&a_{31}+a_{32}x-y\sin \varphi +z\cos \varphi,  \notag
\end{eqnarray}%
where $a_{11}, a_{21}, a_{22}, a_{31}, a_{32}$, and $\varphi$ are real numbers \cite{pav}. 
If the first component of a vector is zero, then the vector is called as isotropic, otherwise it is called non-isotropic vector \cite{pav}.

In $G^{3}$, the scalar product of two vectors $\mathbf{v}=(v_{1},v_{2},v_{3})$ and $\mathbf{w}=(w_{1},w_{2},w_{3})$ is defined by
$$\mathbf{v}\cdot _{G}\mathbf{w} = \left\{
\begin{array}{lr}
v_{1}w_{1} , &  \text{if } v_{1}\neq 0 \text{ or } w_{1}\neq 0\, \ \  \ \\
v_{2}w_{2}+v_{3}w_{3} ,&  \text{if } v_{1}=0 \text{ and } w_{1}=0\,.
\end{array}\right.$$
The Galilean cross product  of these vectors is defined by
\begin{eqnarray}
	\mathbf v\times _{G}\mathbf w=%
	\begin{vmatrix}
		0 & \mathbf e_{2} &\mathbf {e_{3}} \\ 
		v_{1} & v_{2} & v_{3} \\ 
		w_{1} & w_{2} & w_{3}%
	\end{vmatrix}.%
\end{eqnarray}
If $\mathbf{v}\cdot _{G}\mathbf{w}=0$, then $\mathbf{v}$ and $\mathbf{w}$
are perpendicular. 
The norm of $\mathbf{v}$ is defined by
$$\Vert \mathbf{v}\Vert_{G}=\sqrt{\vert\mathbf{v}\cdot_{G}\mathbf{v}\vert}.$$
Let $I\subset \mathbb R$ and let $\gamma :I\rightarrow G^{3}$ be a unit speed curve
with curvature $\kappa>0$ and torsion $\tau$.
Then the curve $\gamma$ is defined by
\begin{eqnarray*}
	\gamma \left( x\right) =\left( x,y\left( x\right) ,z\left( x\right) \right) ,
\end{eqnarray*}
and  that the Frenet frame fields are given by
\begin{eqnarray}{\label{fframe}}
	T\left(x\right) &=&\alpha ^{\prime }\left( x\right), 
	\notag \\
	N\left( x\right) &=& \frac{\gamma''(x)}{\Vert \gamma''(x)\Vert_{G}} 
	\\
	B\left( x\right) &=&T(x)\times _{G}B(x) \notag \\&=&\frac{1}{\kappa \left( x\right) }\left( 0,
	-z^{\prime \prime }\left( x\right) , y^{\prime \prime }\left(
	x\right) \right), \notag
\end{eqnarray}%
where
\begin{equation}
	\kappa \left( x\right) ={\Vert \gamma''(x) \Vert }_{G} \quad \text{and} { \ \ }\tau
	\left( x\right) =\frac{\det \left( \gamma ^{\prime }\left( x\right) ,\gamma
		^{\prime \prime }\left( x\right) ,\gamma ^{\prime \prime \prime }\left(
		x\right) \right) }{\kappa ^{2}\left( x\right) }\,.
\end{equation}%

%
The vector fields $\mathbf{T, N} $ and $\mathbf{B}$ are called the tangent vector field, the principal normal
and the binormal vector field, respectively \cite{pav}. Therefore, the Frenet-Serret formulas can be written in matrix form as
\begin{eqnarray}
	\begin{bmatrix}
		\mathbf{T} \\ 
		\mathbf{N} \\ 
		\mathbf{B}%
	\end{bmatrix}%
	^{\prime }=%
	\begin{bmatrix}
		0 & \kappa & 0 \\ 
		0 & 0 & \tau \\ 
		0 &- \tau & 0%
	\end{bmatrix}%
	\begin{bmatrix}
		\mathbf{T} \\ 
		\mathbf{N} \\ 
		\mathbf{B}%
	\end{bmatrix}\,.
\end{eqnarray}%

There is another useful frame for study curves on a surface. For an easy reference we call this surface $M$. This frame can be formed by two basic vectors. These vectors are a unit tangent vector field $\mathbf{T}$ of the curve $\gamma$ on $M$ and the unit normal vector field $\mathbf{n}$ of $M$ at the point $\gamma(x)$ of $\gamma$. Therefore, the frame field $\{\mathbf{T, Q, n}\}$ is obtained and is called Darboux frame or the tangential-normal frame field. Here, $\mathbf{Q=n}\times_{G}\mathbf{T}$.
%


\begin{theorem}Let  $\gamma :I\subset \mathbb{R}\rightarrow M\subset G^{3}$ be a unit-speed curve, and let $\{\mathbf{T, Q, n}\}$ be the Darboux frame field of $\gamma$ with respect to M. Then the Frenet formulas in matrix form is given by
	\begin{eqnarray}\label{Darboux}
		\begin{bmatrix}
			\mathbf{T} \\ 
			\mathbf{Q} \\ 
			\mathbf{n}
		\end{bmatrix}
		^{\prime }=
		\begin{bmatrix}
			0 & \kappa_g & \kappa_n \\ 
			0 & 0 & \tau_g \\ 
			0 & -\tau_g & 0
		\end{bmatrix}
		\begin{bmatrix}
			\mathbf{T} \\ 
			\mathbf{Q} \\ 
			\mathbf{n}
		\end{bmatrix}\, ,
	\end{eqnarray}
	where $\kappa_g$, $\kappa_n$ and $\tau_g$ are called geodesic curvature, normal curvature and geodesic torsion, respectively.
\end{theorem}

\begin{proof}	It follows from solving \eqref{Darboux} componentwise \cite{buket, tevfik} .
\end{proof}

Also, (\ref{Darboux}) implies the important relations
\begin{eqnarray}\label{kt}
	\kappa^2(x)=\kappa^2_g(x)+\kappa^2_n(x),  \hskip .5cm \tau(x)=-\tau_g(x)+\frac{\kappa'_g(x)\kappa_n(x)-\kappa_g(x)\kappa'_n(x)}{\kappa^2_g(x)+\kappa^2_n(x)}
\end{eqnarray} 
where $\kappa(x)$ and $\tau(x)$ are the curvature and the torsion of $\beta$, respectively.
We refer to \cite{pav, ros, yag} for detailed treatment of Galilean and pseudo-Galilean geometry.
\section{Special Smarandache Curves with Darboux Apparatus with Respect to Frenet Frame in $G_{3}$}

In this section, we will give special Smarandache curves with Darboux apparatus with respect to Frenet frame of a curve on a surface in $G_3$. In order to the position vector of an arbitrary curve with geodesic curvature $\kappa_{g}$, normal curvature $\kappa_{n}$ and geodesic torsion $\tau_{g}$ on the surface in $G_{3}$ \cite{buket}. 

Based on the definition of Smarandache curve in \cite{saad,suha}, we will state the following definition.
\begin{definition}\label{smadef}
	Let $\gamma(x)$ be a unit speed curve in $G_3$ and ${\mathbf{T, N, B}}$ be the Frenet frame field along with $\gamma$. Special Smarandache $\mathbf{TN, TB}$ and $\mathbf{TNB}$ curves are, respectively, defined by
	\begin{eqnarray}
		\gamma_{\mathbf{TN}}&=&\mathbf{T}+\mathbf{N}\\
		\gamma_{\mathbf{TB}}&=&\mathbf{T}+\mathbf{B}\\
		\gamma_{\mathbf{TNB}}&=&\mathbf{T}+\mathbf{N}+\mathbf{B}.
	\end{eqnarray}
\end{definition}

The following result which is stated as theorem is our main work in this article.
\begin{theorem}\label{posmat}
	
	The $\mathbf{TN}$, $\mathbf{TB}$ and $\mathbf{TNB}$ special Smarandache curves with Darboux apparatus of $\gamma$ with respect to Frenet frame are, respectively, written as
	
	\begin{eqnarray}{\label{posma}}
		\gamma_{\mathbf{TN}}&=&\left( 
		\begin{array}{c}
			1,\int N_1 dx+ \frac{1}{\sqrt{{\kappa_g}^2+{\kappa_n}^2}}N_1, \, 
			\int N_2 dx+\frac{1}{\sqrt{{\kappa_g}^2+{\kappa_n}^2}}N_2
		\end{array}\notag
		\right)  \\\notag
		&& \\
		\gamma_{\mathbf{TB}} &=&\left( 
		\begin{array}{c}
			\ 1\ ,\ \int N_1 dx - \frac{1}{\sqrt{{\kappa_g}^2+{\kappa_n}^2}}N_2,\, 
			\int N_2 dx+\frac{1}{\sqrt{{\kappa_g}^2+{\kappa_n}^2}}N_1
		\end{array}
		\right)  \\\notag
		&& \\\notag
		\gamma_{\mathbf{TNB}} &=&\ \left( 
		\begin{array}{c}
			1, \, \int N_1 dx+\frac{1}{\sqrt{{\kappa_g}^2+{\kappa_n}^2}}(N_{1}-N_{2}),\,
			\int N_2 dx+\frac{1}{\sqrt{{\kappa_g}^2+{\kappa_n}^2}}(N_{1}+N_{2})%
		\end{array}%
		\right)
	\end{eqnarray}
	where
	\begin{eqnarray*}
		{N_1}&=&\kappa _{g}\sin \Big(\int\tau _{g}dx\Big)+\kappa _{n}\cos
		\Big(\int \tau _{g}dx\Big),\\ 
		N_2&=&\kappa _{g}\cos \Big(\int \tau _{g}dx\Big)-\kappa _{n}\sin \Big(\int \tau
		_{g}dx\Big). 
	\end{eqnarray*} 
\end{theorem}
\begin{proof}
	The position vector of an arbitrary curve with geodesic curvature $\kappa_{g}$, normal curvature $\kappa_{n}$ and geodesic torsion $\tau_{g}$ on the surface in $G_{3}$ which is introduced by \cite{buket} as follows
	\begin{eqnarray}\label{pos}
		\gamma(x)\ =\left(
		\begin{array}{c}
			x,\, \int (\int (\kappa _{g}(x)\sin (\int \tau _{g}(x)dx)-\kappa _{n}(x)\int \tau
			_{g}(x)\sin (\int \tau _{g}(x)dx)dx)dx)dx,\\ 
			\\ 
			\int (\int (\kappa _{g}\cos (\int \tau _{g}dx)-\kappa _{n}\int \tau _{g}\cos
			(\int \tau _{g}dx)dx)dx)dx
		\end{array}
		\right)
	\end{eqnarray}
	The derivatives of this curve are, respectively, given by;
	\begin{eqnarray}
		\notag
		\gamma^{\prime}(x) &=&\left( 
		\begin{array}{c}
			1, \, \int (\kappa _{g}\sin (\int \tau _{g}dx)-\kappa _{n}\int \tau _{g}\sin
			(\int \tau _{g}ds)dx)dx, \\ 
			\\ 
			\ \int (\kappa _{g}\cos (\int \tau _{g}dx)-\kappa _{n}\int \tau _{g}\cos(\int \tau _{g}dx)dx)dx%
		\end{array}
		\right)\notag \\
		&&  \notag \\
		\gamma^{\prime \prime }(x) &=&\left( 
		\begin{array}{c}
			0, \, \kappa _{g}\sin (\int \tau _{g}dx)-\kappa _{n}\int \tau _{g}\sin
			(\int \tau _{g}dx)dx, \\ 
			\\ 
			\kappa _{g}\cos (\int \tau _{g}dx)-\kappa _{n}\int \tau _{g}\cos (\int \tau
			_{g}dx)dx
		\end{array}
		\right)  \notag
	\end{eqnarray}
	The Frenet frame vector fields with Darboux apparatus of $\gamma$ are determined as follows 
	\begin{eqnarray*}
		\mathbf{T} &=&\left( 
		\begin{array}{c}
			1, \, \int (\kappa _{g}\sin (\int \tau _{g}dx)-\kappa _{n}\int \tau _{g}\sin
			(\int \tau _{g}dx)dx)dx,  \\ 
			\\  
			\ \int (\kappa _{g}\cos (\int \tau _{g}dx)-\kappa _{n}\int \tau _{g}\cos
			(\int \tau _{g}dx)dx)dx%
		\end{array}%
		\right) \\ 
		&&	\\
		\mathbf{N} &=&\frac{1}{\sqrt{{\kappa_g}^2+{\kappa_n}^2}}\left(
		\begin{array}{c}
			0,\, \kappa _{g}\sin(\int \tau _{g}dx)+\kappa _{n}\cos(\int \tau _{g}dx),
			\\ \\ \kappa _{g}\cos(\int \tau _{g}dx)-\kappa _{n}\sin(\int \tau _{g}dx)
		\end{array}
		\right)\\
		&& \\
		\mathbf{B} &=&\frac{1}{\sqrt{{\kappa_g}^2+{\kappa_n}^2}}\left(
		\begin{array}{c}
			0,\, -\kappa _{g}\cos(\int \tau _{g}dx)+\kappa _{n}\sin(\int \tau _{g}dx),
			\\ \\ \kappa _{g}\sin(\int \tau _{g}dx)+\kappa _{n}\cos(\int \tau _{g}dx)
		\end{array}
		\right)
	\end{eqnarray*}
	Using the definition (\ref{smadef}), we obtain desired results.
	
	We now provide some applications of this theorem for some special curves.
\end{proof}
\section{Applications}
We begin with studying Smarandache curves of important special curves lying on surfaces such as geodesic, asymtotic and curvature (or principal) line. Also, we will provide special cases such as helix and Salkowski curve of these curves.

Let $\gamma$ be regular curve on a surface in $G^3$ with the curvature $\kappa$, the torsion $\tau$, the geodesic curvature $\kappa_g$, the normal curvature $\kappa_n$ and the geodesic torsion $\tau_g$. 

\begin{definition}\label{defgap}
	\cite{krey} We can say that $\gamma$ is
	\begin{eqnarray*}
		\begin{split}
			geodesic \, curve  &\Longleftrightarrow \kappa_g\equiv 0,
			\\asymptotic \, curve & \Longleftrightarrow \kappa_n\equiv 0,
			\\line \, of \, curvature  & \Longleftrightarrow \tau_g\equiv 0.
		\end{split}
	\end{eqnarray*}
	Also, We can say that $\gamma$ is called:
	\begin{eqnarray}\label{helsal}
		\begin{array}{ccc}
			\kappa, \tau  & \hskip 1cm& r\\
			\hline
			\kappa\equiv0 &\Longleftrightarrow &\textbf{a straight line.}\\
			\tau\equiv0 &\Longleftrightarrow &\textbf{a plane curve.}\\
			\kappa\equiv\textit{cons.$>$0},\tau\equiv\textit{cons.$>$0} &\Longleftrightarrow &\textbf{a circular helix or W-curve.}\\
			\frac{\tau}{\kappa}\equiv\textit{cons.} &\Longleftrightarrow &\textbf{a generalized helix.}\\
			\kappa\equiv\textit{cons.}, \tau\neq\textit{cons.} &\Longleftrightarrow &\textbf{Salkowski curve \cite{mon,sal}.}\\
			\kappa\neq\textit{cons.}, \tau\equiv\textit{cons.} &\Longleftrightarrow &\textbf{anti-Salkowski curve \cite{sal}.}\\
		\end{array}
	\end{eqnarray}
\end{definition}
\subsection{The position vectors of Smarandache curves of a general geodesic curve in $G_3$}

\begin{theorem}{\label{thmgeo}}
	The position vectors $\alpha_g(x)$ of Smarandache curves of a family of geodesic curve in $G_3$ are provided by

	\begin{eqnarray*}
		\alpha^g _{\mathbf{TN}} &=&\left( 
		\begin{array}{c}
			1, \, \int \kappa _{n}\cos (\int \tau _{g}dx)dx+\cos (\int \tau _{g}dx),
			-\int \kappa _{n}\sin (\int \tau _{g}dx)dx-\sin (\int \tau _{g}dx)
		\end{array}
		\right) \\
		&&	\\
		\alpha^g _{\mathbf{TB}} &=&\left( 
		\begin{array}{c}
			1, \, \int \kappa _{n}\cos (\int \tau _{g}dx)dx+\sin (\int \tau _{g}dx),
			-\int \kappa _{n}\sin (\int \tau _{g}dx)dx+\cos (\int \tau _{g}dx)
		\end{array}
		\right)\\
		&&\\
		\alpha^g _{\mathbf{TNB}} &=&\left( 
		\begin{array}{c}
			1, \, \int \kappa _{n}\cos (\int \tau _{g}dx)dx+\cos (\int \tau _{g}dx)+\sin (\int \tau _{g}dx),\\
			\\
			-\int \kappa _{n}\sin (\int \tau _{g}dx)dx+\cos (\int \tau _{g}dx)-\sin (\int \tau _{g}dx)
		\end{array}%
		\right)  \notag
	\end{eqnarray*}
\end{theorem}
\begin{proof}
	The above equations are obtained as general position vectors for $\mathbf{TN}, \, \mathbf{TB}$ and $\mathbf{TNB}$ special Smarandache curves with Darboux apparatus of a geodesic curve on a surface in $G_3$ by using the definition (\ref{defgap}) and Theorem \ref{posmat}.
\end{proof}
Now, we will give the position vectors for special Smarandache curves of some special cases of a geodesic curve in $G_3$.
\begin{corollary}
	
	The position vectors of special Smarandache curves of a family of geodesic curve that is a circular helix in $G_3$ are given by the equations
	%
	%
	%
	%
	%
	%
	\begin{eqnarray*}
		\alpha^g_{ch}{_\mathbf{TN}}(x) &=& \left(
		\begin{array}{c}
			1, \, \frac{e}{c}\sin(cx+c_{1})+\frac{e}{c}\cos(cx+c_{1}), \\ \\
			\frac{e}{c}\cos(cx+c_{1})-\frac{e}{c}\sin(cx+c_{1})
		\end{array}
		\right)\\
		\\ 
		\alpha^g_{ch}{_\mathbf{TB}}(x) &=& \left( 
		\begin{array}{c}
			1, \,  (\frac{e+c}{c})\sin(cx+c_{1}), \, (\frac{e+c}{c})\cos(cx+c_{1})
		\end{array}
		\right)\\
		\\ 
		\alpha^g_{ch}{_\mathbf{TNB}}(x) &=&\left(
		\begin{array}{c}
			1, \,
			(\frac{e+c}{c})\sin(cx+c_{1})+\cos(cx+c_{1}), \\
			\\
			(\frac{e+c}{c})\cos(cx+c_{1})-\sin(cx+c_{1}) \\
			\\
		\end{array}%
		\right)
	\end{eqnarray*}
	where $c, \, c_1$ and $e$ are integral constants.
\end{corollary}

\begin{corollary}
	The position vectors of special Smarandache curves of a family of geodesic curve that is a generalized helix in $G_3$ are given by the equations
	%
	%
	%
	%
	%
	\begin{eqnarray*}
		\alpha^g_{gh}{_\mathbf{TN}}(x)&=&\left(
		\begin{array}{c}
			1, \, \frac{1}{d}\sin(d\int\kappa_{n}dx)+\cos(d\int\kappa_{n}dx), \\
			\\ \frac{1}{d}\cos(d\int\kappa_{n}dx)-\sin(d\int\kappa_{n}dx)\\
		\end{array}
		\right)\\
		\\ 
		\alpha^g_{gh}{_\mathbf{TB}}(x)&=&\left(
		\begin{array}{c}
			1, \, \frac{1}{d}\sin(d\int\kappa_{n}dx)+\sin(d\int\kappa_{n}dx), \\
			\\ \frac{1}{d}\cos(d\int\kappa_{n}dx)+\cos(d\int\kappa_{n}dx)\\
		\end{array}%
		\right) \\
		\\ 
		\alpha^g_{gh}{_\mathbf{TNB}}(x)&=&\left(
		\begin{array}{c}
			1, \, \frac{d+1}{d}\sin(d\int\kappa_{n}dx)+\cos(d\int\kappa_{n}dx), \\
			\\
			\frac{d+1}{d}\cos(d\int\kappa_{n}dx)-\sin(d\int\kappa_{n}dx) \\
		\end{array}
		\right)
	\end{eqnarray*}
	where $d$ is integral constant.
	%
\end{corollary}

\begin{corollary}
	The position vectors of Smarandache curves of a family of geodesic that is a Salkowski curve in $G_3$ are given by the equations
	%
	%
	%
	%
	%
	%
	\begin{eqnarray*}
		\alpha^g_{s}{_\mathbf{TN}}(x)&=&\left(
		\begin{array}{c}
			1, \, m\int\cos(\int\tau_{g}dx)dx+\cos(\int\tau_{g}dx), \\
			\\ -m\int\sin(\int\tau_{g}dx)dx-\sin(\int\tau_{g}dx) \\
		\end{array}%
		\right) \\
		\\ 
		\alpha^g_{s}{_\mathbf{TB}}(x)&=&\left(
		\begin{array}{c}
			1, \,  m\int\cos(\int\tau_{g}dx)dx+\sin(\int\tau_{g}dx), \\
			\\ -m\int\sin(\int\tau_{g}dx)dx+\cos(\int\tau_{g}dx)  \\
		\end{array}%
		\right)\\
		\\ 
		\alpha^g_{s}{_\mathbf{TNB}}(x)&=&\left(
		\begin{array}{c}
			1, \, m\int\cos(\int\tau_{g}dx)dx+\cos(\int\tau_{g}dx)+\sin(\int\tau_{g}dx),\\
			\\ -m\int\sin(\int\tau_{g}dx)dx+\cos(\int\tau_{g}dx)-\sin(\int\tau_{g}dx) \\
		\end{array}%
		\right)
	\end{eqnarray*}
	where $m$ is an integral constant.
\end{corollary}
\begin{corollary}
	The position vectors of Smarandache curves of a family of geodesic that is a anti-Salkowski curve in $G_3$ are given by the equations
	%
	%
	%
	%
	%
	\begin{eqnarray*}
		\alpha^g_{as}{_\mathbf{TN}}(x)&=&\left(
		\begin{array}{c}
			1, \, \int\kappa_{n}\cos(cx+c_{1})dx+\cos(cx+c_{1}),\\
			\\ -\int\kappa_{n}\sin(cx+c_{1})dx-\sin(cx+c_{1})
		\end{array}%
		\right)\\
		\\ 
		\alpha^g_{as}{_\mathbf{TB}}(x)&=&\left(
		\begin{array}{c}
			1, \, \int\kappa_{n}\cos(cx+c_{1})dx+\sin(cx+c_{1}) , \\
			\\ -\int\kappa_{n}\sin(cx+c_{1})dx+\cos(cx+c_{1})
		\end{array}%
		\right)\\	
		\alpha^g_{as}{_\mathbf{TNB}}(x)&=&\left(
		\begin{array}{c}
			1, \, \int\kappa_{n}\cos(cx+c_{1})dx+\cos(cx+c_{1})+\sin(cx+c_{1}), \\
			\\ -\int\kappa_{n}\sin(cx+c_{1})dx+\cos(cx+c_{1})-\sin(cx+c_{1})
		\end{array}%
		\right)
	\end{eqnarray*}
	where $c$ and $c_{1}$ are integral constants.
\end{corollary}

We want to note that above corollaries can be proved by using the equations (\ref{kt}),  (\ref{helsal}) and Theorem \ref{thmgeo}.

\subsection{The position vectors of Smarandache curves of an general asymptotic curve in $G_3$}

\begin{theorem}{\label{thmasy}}
	The position vectors $\alpha_g(x)$ of Smarandache curves of a family of asymptotic curve in $G_3$ are provided by
	%
	%
	%
	%
	%
	%
	%
	%
	\begin{eqnarray*}
		\beta^a _{\mathbf{TN}} &=&\left( 1, \, \int \kappa _{g}\sin (\int \tau _{g}dx)dx+\sin
		\int \tau _{g}dx\ ,\ \int \kappa _{g}\cos (\int \tau _{g}dx)dx+\cos \int
		\tau _{g}dx\right) \\
		&&  \notag \\
		\beta^a _{\mathbf{TB}} &=&\left( 1, \, \int \kappa _{g}\sin (\int \tau _{g}dx)dx-\cos
		\int \tau _{g}dx\ ,\ \int \kappa _{g}\cos (\int \tau _{g}dx)dx+\sin \int
		\tau _{g}dx\right)  \notag \\
		&&  \notag \\
		\beta^a _{\mathbf{TNB}} &=&\left( 
		\begin{array}{c}
			1, \,  \int \kappa _{g}\sin (\int \tau _{g}dx)dx+\sin \int \tau _{g}dx-\cos
			\int \tau _{g}dx, \, \\ \\
			\int \kappa _{g}\cos (\int \tau _{g}dx)dx+\cos \int \tau _{g}dx+\sin \int
			\tau _{g}ds%
		\end{array}%
		\right)  \notag
	\end{eqnarray*}
\end{theorem}
\begin{proof}
	By using the definition (\ref{defgap}) in Theorem \ref{posmat}, then the above equations are obtained as general position vectors for $\mathbf{TN}, \, \mathbf{TB}$ and $\mathbf{TNB}$ special smarandache curves with Darboux apparatus of an asymptotic curve on a surface in $G_3$.
\end{proof}
%
%
%
Now, we will give the position vectors for Smarandache curves of some special cases of an asymptotic curve in $G_3$
\begin{corollary}
	The position vectors of Smarandache curves of a family of asymptotic curve that is a circular helix in $G_3$ are given by the equations 
	%
	%
	%
	%
	%
	\begin{eqnarray*}
		\beta^a_{ch}{_\mathbf{TN}}(x)&=&\left(
		\begin{array}{c}
			1, \, -\frac{f}{c}\cos(cx+c_{1})+\sin(cx+c_{1})\ , \\
			\\ \frac{f}{c}\sin(cx+c_{1})+\cos(cx+c_{1}) \\
		\end{array}%
		\right)\\ \\
		\beta^a_{ch}{_\mathbf{TB}}(x)&=&\left(
		\begin{array}{c}
			1, \,  -(\frac{c+f}{c})\cos(cx+c_{1}),
			\,  (\frac{c+f}{c})\sin(cx+c_{1}) \\
		\end{array}
		\right)\\ \\
		\beta^a_{ch}{_\mathbf{TNB}}(x)&=&\left(
		\begin{array}{c}
			1, \,  -(\frac{c+f}{c})\cos(cx+c_{1})+\sin(cx+c_{1}),\\
			\\  (\frac{c+f}{c})\sin(cx+c_{1})+\cos(cx+c_{1}) 
		\end{array}%
		\right)
	\end{eqnarray*}
	where $c, c_{1}$ and $f$ are integral constants.
\end{corollary}
\begin{corollary}
	The position vectors of Smarandache curves of a family of asymptotic curve that is a generalized helix in $G_3$ are given by the equations 
	%
	%
	%
	%
	%
	%
	\begin{eqnarray*}
		\beta^a_{gh}{_\mathbf{TN}}(x)&=&\left(
		\begin{array}{c}
			1, \, -\cos(k\int \kappa_{g}dx)+\sin(k \int \kappa_{g}dx), \\
			\\
			\sin(k\int \kappa_{g}dx)+\cos(k \int \kappa_{g}dx) \\
		\end{array}%
		\right) \\
		\\
		\beta^a_{gh}{_\mathbf{TB}}(x)&=&\Big(
		\begin{array}{c}
			1, \,  -2\cos(k\int \kappa_{g}dx), \,
			2\sin(k\int \kappa_{g}dx)	\\
		\end{array}%
		\Big)\\
		\\
		\beta^a_{gh}{_\mathbf{TNB}}(x)&=&\left(
		\begin{array}{c}
			1, \, -2\cos(k\int \kappa_{g}dx)+\sin(k\int \kappa_{g}dx), \\
			\\
			2\sin(k\int \kappa_{g}dx)+\cos(k\int \kappa_{g}dx)	\\
		\end{array}%
		\right)
	\end{eqnarray*}
	where $k$ is integral constant.
\end{corollary}
\begin{corollary}
	The position vectors of Smarandache curves of a family of asymptotic curve that is a Salkowski curve in $G_3$ are given by the equations 
	%
	%
	%
	%
	%
	%
	\begin{eqnarray*}
		\beta^a_{s}{_\mathbf{TN}}(x)&=&\left(
		\begin{array}{c}
			1, \,  \int(f\sin(\int \tau_{g}dx))dx+\sin(\int \tau_{g}dx), \\
			\\
			\int(f\cos(\int \tau_{g}dx))dx+\cos(\int \tau_{g}dx)		\\
		\end{array}%
		\right)\\
		\\
		\beta^a_{s}{_\mathbf{TB}}(x)&=&\left(
		\begin{array}{c}
			1, \,  \int(f\sin(\int \tau_{g}dx))dx-\cos(\int \tau_{g}dx), \\
			\\
			\int(f\cos(\int \tau_{g}dx))dx+ \sin(\int \tau_{g}dx)	\\
		\end{array}%
		\right)\\
		\\
		\beta^a_{s}{_\mathbf{TNB}}(x)&=&\left(
		\begin{array}{c}
			1, \,  \int(f\sin(\int \tau_{g}dx))dx+\sin(\int \tau_{g}dx)-\cos(\int \tau_{g}dx), \\
			\\
			\int(f\cos(\int \tau_{g}dx))dx+\cos(\int \tau_{g}dx)+ \sin(\int \tau_{g}dx)	\\
		\end{array}%
		\right)
	\end{eqnarray*}
	where $f$ is constant.
\end{corollary}
\begin{corollary}
	The position vectors of Smarandache curves of a family of asymptotic curve that is an anti-Salkowski curve in $G_3$ are given by the equations 
	%
	%
	%
	%
	%
	%
	%
	%
	%
	\begin{eqnarray*}
		\beta^a_{as}{_\mathbf{TN}}(x)&=&\left(
		\begin{array}{c}
			1, \,  \int(\kappa_{g}\sin(cx+c_{1}))dx+\sin(cx+c_{1}),\\ \\ \int(\kappa_{g}\cos(cx+c_{1}))dx+\cos(cx+c_{1})	\\
		\end{array}
		\right)\\
		\\
		\beta^a_{as}{_\mathbf{TB}}(x)&=&\left(
		\begin{array}{c}
			1, \, \int(\kappa_{g}\sin(cx+c_{1}))dx-\cos(cx+c_{1}), \\ \\
			\int(\kappa_{g}\cos(cx+c_{1}))dx+\sin(cx+c_{1})	\\
		\end{array}
		\right)\\
		\\
		\beta^a_{as}{_\mathbf{TNB}}(x)&=&\left(
		\begin{array}{c}
			1, \, \int(\kappa_{g}\sin(cx+c_{1}))dx+\sin(cx+c_{1})-\cos(cx+c_{1}), \\
			\\
			\int(\kappa_{g}\cos(cx+c_{1}))dx+\cos(cx+c_{1})+\sin(cx+c_{1})
		\end{array}
		\right)
	\end{eqnarray*}
	where $c$ and $c_1$ are constants.
\end{corollary}
\begin{proof}
	We want to point out that above corollaries can be proved by using the equations (\ref{kt}), (\ref{helsal}) and Theorem \ref{thmasy}.
\end{proof}
\subsection{The position vectors of Smarandache curves of a general curvature line in $G_3$}

\begin{theorem}{\label{thmcur}}
	The position vectors $\gamma^c(x)$ of Smarandache curves of a family of curvature line in $G_3$ are provided by
	%
	%
	%
	%
	%
	%
	%
	%
	%
	\begin{eqnarray*}
		\gamma^c _{_\mathbf{TN}} &=&\left(
		\begin{array}{c}
			1,\, \int (\kappa _{g}\sin a + \kappa_n \cos a)dx+\frac{1}{\sqrt{\kappa_g^2+\kappa_n^2}}(\kappa _{g}\sin a + \kappa_n \cos a),\\
			\, \int (\kappa _{g}\cos a - \kappa_n \sin a)dx-\frac{1}{\sqrt{\kappa_g^2+\kappa_n^2}}(\kappa _{g}\cos a - \kappa_n \sin a)
		\end{array}
		\right) \\
		&&  \notag \\
		\gamma^c _{_\mathbf{TB}} &=&\left( 
		\begin{array}{c}
			1,\, \int (\kappa _{g}\sin a + \kappa_n \cos a)dx-\frac{1}{\sqrt{\kappa_g^2+\kappa_n^2}}(\kappa _{g}\cos a - \kappa_n \sin a),\\
			\, \int (\kappa _{g}\cos a - \kappa_n \sin a)dx+\frac{1}{\sqrt{\kappa_g^2+\kappa_n^2}}(\kappa _{g}\sin a + \kappa_n \cos a)
		\end{array}
		\right) \\
		&& \notag \\
		\gamma^c _{_\mathbf{TNB}} &=&\left( 
		\begin{array}{c}
			1,\, \int (\kappa _{g}\sin a + \kappa_n \cos a)dx\\-\frac{1}{\sqrt{\kappa_g^2+\kappa_n^2}}(\kappa _{g}(\sin a-\cos a) + \kappa_n(\sin a+\cos a)),\\
			\, \int (\kappa _{g}\cos a - \kappa_n \sin a)dx\\+\frac{1}{\sqrt{\kappa_g^2+\kappa_n^2}}(\kappa _{g}(\sin a+\cos a) + \kappa_n (\cos a-\sin a))
		\end{array}
		\right)  \notag
	\end{eqnarray*}
\end{theorem}
\begin{proof}
	By using the definition (\ref{defgap}) in Theorem \ref{posmat}, then the above equations are obtained as general position vectors for $\mathbf{TN}, \, \mathbf{TB}$ and $\mathbf{TNB}$ special smarandache curves with Darboux apparatus of a curvature line on a surface in $G_3$.
\end{proof}
%
Now, we will give the position vectors for Smarandache curves of some special cases of a curvature line in $G_3$
\begin{corollary}
	The position vectors $\gamma^c(x)$ of Smarandache curves of a family of curvature line with $\kappa_{g}\equiv$ const. and $\kappa_{n}\equiv$ const. is a circular helix in $G_3$ are provided by
	%
	%
	%
	%
	\begin{eqnarray*}
		\gamma^c_{ch}{_\mathbf{TN}}(x)&=&\left(
		\begin{array}{c}
			1, \, (a_{1}\sin a+a_{2}\cos a)x+\frac{1}{\sqrt{a_1^2+a_2^2}}(a_{1}\sin a+a_2\cos a), \\
			\\ (a_{1}\cos a-a_{2}\sin a)x-\frac{1}{\sqrt{a_1^2+a_2^2}}(a_{1}\cos a-a_2\sin a)
		\end{array} \right) \\
		\\
		\gamma^c_{ch}{_\mathbf{TB}}(x)&=&\left(
		\begin{array}{c}
			1, \, (a_{1}\sin a+a_{2}\cos a)x-\frac{1}{\sqrt{a_1^2+a_2^2}}(a_{1}\cos a-a_2\sin a), \\
			\\ (a_{1}\cos a-a_{2}\sin a)x+\frac{1}{\sqrt{a_1^2+a_2^2}}(a_{1}\sin a+a_2\cos a)
		\end{array}
		\right)\\
		\\
		\gamma^c_{ch}{_\mathbf{TNB}}(x)&=&\left(
		\begin{array}{c}
			1, \, (a_{1}\sin a+a_{2}\cos a)x\\+\frac{1}{\sqrt{a_1^2+a_2^2}}(a_{1}(\sin a-\cos a)+a_2(\cos a+\sin a)), \\
			\\ (a_{1}\cos a-a_{2}\sin a)x\\-\frac{1}{\sqrt{a_1^2+a_2^2}}(a_{1}(\sin a+\cos a)-a_2(\cos a-
			\sin a))
		\end{array}
		\right)
	\end{eqnarray*}
	%
\end{corollary}
\begin{proof}
	By using the equations (\ref{kt}) and (\ref{helsal}) in Theorem \ref{thmcur}, we obtain the above equation.
\end{proof}

We will now provide some illustrative examples for arbitrary curve on a surface.
\begin{example}
	In \eqref{pos}, if we let $\kappa_g(x)=\sin x, \, \kappa_n(x)=\cos x$ and $\tau_g(x)=x$, we obtain the following curve:
	\begin{eqnarray}	
		\gamma(x)= \left(
		\begin{array}{c}
			x, \\\\ \sqrt{\pi } \left( x\cos \left( 1/2 \right) -\cos \left( 1/2 \right)  \right) {\it FresnelC} \left( {\frac {x-1}{ \sqrt{\pi }}} \right) \\
			\mbox{}+ \sqrt{\pi } \left( x\sin \left( 1/2 \right) -\sin \left( 1/2 \right)  \right) {\it FresnelS} \left( {\frac {x-1}{ \sqrt{\pi }}} \right) \\
			\mbox{}-\cos \left( 1/2 \right) \sin \left( 1/2\, \left( x-1 \right) ^{2} \right) +\sin \left( 1/2 \right) \cos \left( 1/2\, \left( x-1 \right) ^{2} \right),\\\\
			- \sqrt{\pi } \left( \sin \left( 1/2 \right) -x\sin \left( 1/2 \right)  \right) {\it FresnelC} \left( {\frac {x-1}{ \sqrt{\pi }}} \right) \\
			\mbox{}- \sqrt{\pi } \left( x\cos \left( 1/2 \right) -\cos \left( 1/2 \right)  \right) {\it FresnelS} \left( {\frac {x-1}{ \sqrt{\pi }}} \right) \\-\cos \left( 1/2\, \left( x-1 \right) ^{2} \right) \cos \left( 1/2 \right) 
			\mbox{}-\sin \left( 1/2\, \left( x-1 \right) ^{2} \right) \sin \left( 1/2 \right)
		\end{array}
		\right)
	\end{eqnarray}	
	where $$FresnelS(x)=\int \sin\left(\frac{\pi x^2}{2}\right) dx, \hskip .5cm FresnelC(x)=\int \cos\left(\frac{\pi x^2}{2}\right).$$
	The special Smaradache curves of $\gamma$ can be obtained directly from Definition \ref{smadef}, or replacing $\kappa_g(x)$, $\kappa_n(x)$ and $\tau_g(x)$ by $\sin x, \cos x$ and $x$ in Theorem 3, respectively.
	In this case, the graphs of $\gamma$ curve and its $\mathbf{TN, TB, TNB}$ special Smarandache curves are given as follows Figure 1.
\end{example}
\begin{figure}[h]
	\begin{center}
		\includegraphics[width=0.4 \textwidth]{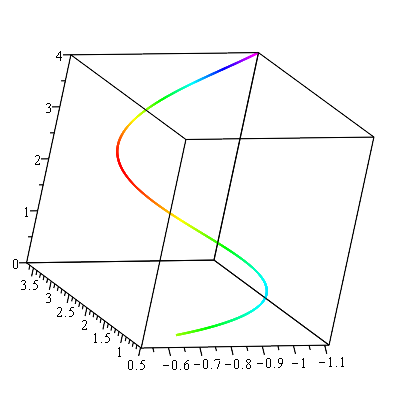}
		\includegraphics[width=0.4\textwidth]{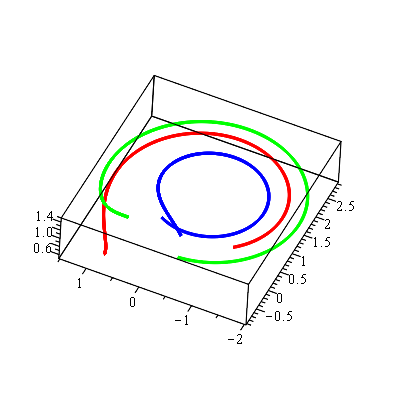}
		\caption{$\gamma$ curve and the right figure is printed from outside to inside  $\gamma_{TNB}, \gamma_{TB}, \gamma_{TN}$ Smarandache curves of $\gamma$} \label{fig1}
	\end{center}
\end{figure}
We now consider another example for geodesic curve on surface along with their graphs.
\begin{example} Let the surface $M$ be defined by
	$$\phi(u,v)=\Bigg(u+v, \frac{u-\sin(u+v)\cos(u+v)}{4}, \frac{\sin(u+v)^2-u^2}{4}\Bigg)$$
	and define the curve $\gamma$ which lies on the surface $M$ as follows
	$$\displaystyle \gamma(x) = \Bigg(x, \frac{x-\sin(x)\cos(x)}{4}, \frac{\sin(x)^2-x^2}{4}\Bigg).$$

	Thus, $\gamma$ is a geodesic curve with  $\kappa(x)=\sin x$ and $\tau(x)\equiv1$ on $M$ in $G_3$.  Also, $\mathbf {T, Q, n}$ vector fields and $\kappa_n(x), \tau_g(x)$ curvatures are obtained by using equation (\ref{Darboux}), (\ref{kt}). Using these curvatures in Theorem \ref{thmgeo}, we derive special Smarandache curves of $\gamma$.
	
\end{example}
\begin{figure}[h]
	\begin{center}
		\includegraphics[width=0.6\textwidth]{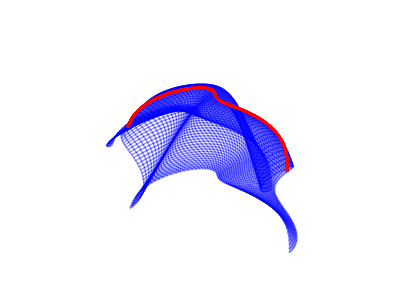}%
		\includegraphics[width=0.4\textwidth]{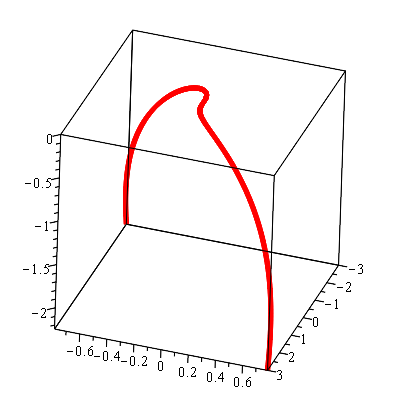}
		\caption{$\phi(u,v)$ surface and $\gamma(x)$ curve}
	\end{center}
\end{figure}

\begin{figure}[h]
	\begin{center}
		\includegraphics[width=0.3\textwidth]{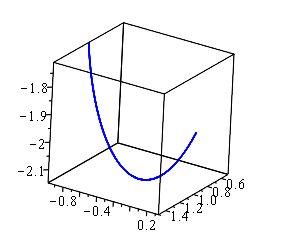}%
		\includegraphics[width=0.3\textwidth]{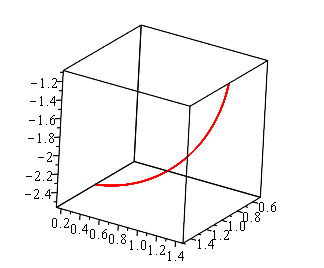}%
		\includegraphics[width=0.3\textwidth]{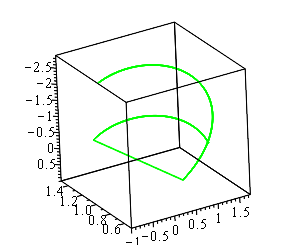}
		\caption{$\gamma_\mathbf{TNB}, \gamma_\mathbf{TB}, \gamma_\mathbf{TN}$ special Smarandache curves of $\gamma$, respectively.} \label{fig1}
	\end{center}
\end{figure}

\begin{figure}[h]
	\begin{center}
		\includegraphics[width=0.3\textwidth]{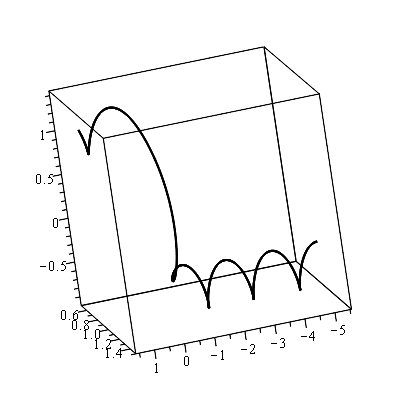}%
		\includegraphics[width=0.3\textwidth]{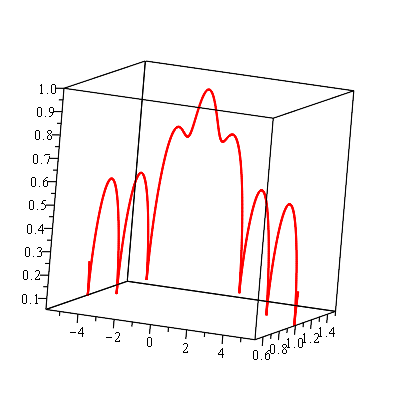}%
		\includegraphics[width=0.3\textwidth]{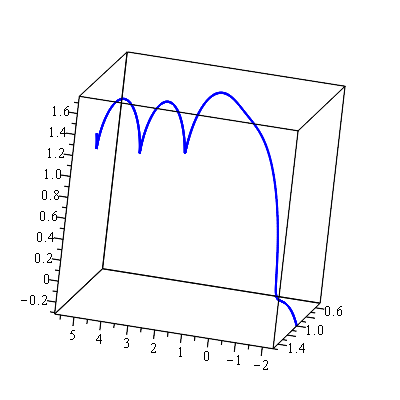}
		\caption{$\gamma_\mathbf{Tn}, \gamma_\mathbf{TQ}, \gamma_\mathbf{TQn}$ special Smarandache curves with respect to Darboux frame of $\gamma$, respectively.} \label{fig1}
	\end{center}
\end{figure}
\newpage
\section{Conclusion}
In this work, we studied general position vectors of special Smarandache curves with Darboux apparatus of an arbitrary curve on a surface in the three-dimensional Galilean space $G^{3}$. As a result of this, we also provided special Smarandache curves of geodesic, asymptotic and curvature line on the surface in $G^{3}$ and provided some related examples of special Smarandache curves with respect to Frenet and Darboux frame of an arbitrary curve on a surface. Finally, we emphasize that one can investigate position vectors of elastic curves on a surface using the general position vectors of curves on  a surface in Galilean space. Last but not least, we want to point out that the results of this study can be easily generalized to families of surfaces that have common Smarandache curves.

\section*{Acknowledgements}

This study was supported financially by the Research Centre of Amasya University (Project No:  FMB-BAP16-0213).

\end{document}